\documentclass[12pt]{article}%
\usepackage{amsmath}
\usepackage{amsfonts}
\usepackage{amssymb}
\usepackage{graphicx}%
\newtheorem{theorem}{Theorem}

\newtheorem{lemma}[theorem]{Lemma}

\newenvironment{proof}[1][Proof]{\noindent\textbf{#1.} }{\ \rule{0.5em}{0.5em}}
\graphicspath{    {converted_graphics/}    {/}}
\begin{document}

\begin{center}
{\Large Dual Nature of Orbits of the Divide-or-Choose 2 Rule: }

{\Large A Quadratic Collatz-type Recursion}

\medskip

H. SEDAGHAT \footnote{Email: hsedagha@vcu.edu}

\end{center}

\begin{abstract}
We obtain a complete characterization of all orbits of a quadratic
Collatz-type recursion called the divide-or-choose-2 rule. Each orbit either
ends in a cycle whose period depends on the initial value or it goes to
infinity. We specify which initial values generate periodic orbits, and also
show that all other initial values generate orbits that go to infinity.

\end{abstract}

\bigskip

Let $\mathbb{N}$\ be the set of all positive integers and $\mathbb{N}%
_{0}=\{0,1,2,3,\ldots\}$ the set of all non-negative integers. In the paper
\cite{HS2} it was shown that all orbits of the map $Q:\mathbb{N}%
_{0}\rightarrow\mathbb{N}_{0}$%
\begin{equation}
Q(n)=\left\{
\begin{array}
[c]{l}%
n/2\quad\text{if }n\text{ is even, i.e. }n\equiv0(\text{mod2)}\\
\binom{n}{2}\text{\quad if }n\text{ is odd, i.e. }n\equiv1(\text{mod2)}%
\end{array}
\right.  \label{dc2}%
\end{equation}
where
\[
\binom{n}{2}=\frac{n(n-1)}{2}%
\]

\noindent is the binomial coefficient, must either end in a cycle or go to
infinity. It was also shown that cycles of all possible lengths were possible;
for instance, if $n=33$ then%
\begin{gather*}
Q(33)=\frac{33(32)}{2}=33(16)=528,\ Q^{2}(33)=Q(Q(33))=\frac{Q(33)}%
{2}=33(8)=264\\
Q^{3}(33)=\frac{33(8)}{2}=132,\ Q^{4}(33)=\frac{33(4)}{2}=66,\ Q^{5}%
(33)=\frac{33(2)}{2}=33
\end{gather*}

Therefore, after 5 iterations of $Q$ the initial value 33 is returned and have
generated a cycle of length 5 which we may write succinctly as%
\[
33\rightarrow528\rightarrow264\rightarrow132\rightarrow66\rightarrow33
\]

We call $Q$ the \textit{divide-or-choose-2 rule}. It is a multiplicative
version of the map%
\[
F(n)=\left\{
\begin{array}
[c]{l}%
n/2\qquad\text{if }n\equiv0(\text{mod2)}\\
\frac{3n-1}{2}\text{\quad\ if }n\equiv1(\text{mod2)}%
\end{array}
\right.
\]
in the following sense:%
\[
\binom{n}{2}=n\left(  \frac{n-1}{2}\right)  ,\qquad\frac{3n-1}{2}=n+\frac
{n-1}{2}%
\]

Equivalently, $F$ is the additive or \textquotedblleft linear" version of $Q$.
We call $F$ and $Q$ \textquotedblleft Collatz-type" functions because $Q$ is
the multiplicative version of $F$ and
\[
F(n)=-T(-n),\quad n\in\mathbb{N}%
\]
where $T$ is the well-known \textit{Collatz, or 3n+1 function} (compressed
form)%
\[
T(n)=\left\{
\begin{array}
[c]{l}%
n/2\qquad\text{if }n\equiv0(\text{mod2)}\\
\frac{3n+1}{2}\text{\quad\ if }n\equiv1(\text{mod2)}%
\end{array}
\right.
\]

Thus the orbits of $F$ in $\mathbb{N}$ are the negatives of the orbits of $T$
in $-\mathbb{N}$. While all orbits of $T$ are conjectured to reach the cycle
$\{1,2,1,2,\ldots\}$ from any initial value $n\in\mathbb{N}$ the orbits of $T$
in $-\mathbb{N}$ include two known nontrivial cycles. A substantial amount of
research has been done on the 3n+1 map and its variants like $F$; see e.g.
\cite{LG}. The existing research shows that $T$ is unlikely to have any cycles
in $\mathbb{N}$ other than the base cycle above and further, $T$ is
unlikely to have a divergent orbit that goes to infinity. However, neither of
these facts have been proved.

\medskip

$Q$ is expected to have divergent orbits since its odd part essentially
squares a number while its even part merely divides it by 2. But as shown in
\cite{HS2} proving the existence of a divergent orbit is not straightforward
$Q$ or other higher degree Collatz-type maps. The goal of this paper is to
prove that the orbits of $Q$ that go to infinity not only exist but they also
constitute almost all of the orbits in the sense of natural density because
the bounded orbits of $Q$ occur only from a limited range of values. This
result answers the open questions in \cite{HS2} and completes our study of the
map $Q$.

\medskip

We begin with the observation that $Q$ has no nontrivial orbits in
$-\mathbb{N}$ since its odd half is always non-negative. So we consider only
the orbits of $Q$ in $\mathbb{N}$. 

For each $x_{0}\in\mathbb{N}_{0}$ recall that the numbers%
\[
x_{0},Q(x_{0}),Q(Q(x_{0})),\ldots
\]
constitute an \textit{orbit} or \textit{trajectory} in $\mathbb{N}_{0}$ of the
recursion%
\begin{equation}
x_{n+1}=Q(x_{n}) \label{q1}%
\end{equation}

If $x_{m}=x_{0}$ for some $m\in\mathbb{N}$ then the numbers $x_{0}%
,x_{1},\ldots,x_{m-1}$ repeat so we have a \textit{periodic orbit with period}
$m$ or equivalently, an $m$-\textit{cycle}, i.e. a cycle of length $m$. A
number $x_{0}$ that lies on a cycle is called a \textit{periodic point} of
$Q$. If $m=1$ then the cycle is often called a \textit{fixed point} of $Q$. It
is straightforward to show that $Q$ has two fixed points in $\mathbb{N}_{0}$,
namely, 0 and 3.

\medskip

An orbit $x_{0},x_{1},\ldots,x_{n},\ldots$ \textit{goes to infinity} (or
\textit{escapes to infinity} or \textit{diverges}) if for every positive
integer $M$ the value of $x_{n}$ exceeds $M$ for all large enough indices $n$.
Orbits that go to infinity are unbounded and have no bounded subsequences.
While not monotone in general, arbitrarily long monotone chains of odd or even
numbers may well exist in some orbits; see \cite{HS2}.

\medskip

If $x_{0}=2^{m}k$ is an arbitrary even number where $m\in\mathbb{N}_{0}$ and
$k$ is odd then%
\begin{equation}
x_{1}=2^{m-1}k,\ x_{2}=2^{m-2}k,\ \ldots,\ x_{m}=k\label{evo}%
\end{equation}

Also every odd number larger than 1 can be written as $2^{m}k+1$ where
$m\in\mathbb{N}_{0}$ and $k$ is odd. Note that
\begin{equation}
Q(2^{m}k+1)=(2^{m}k+1)\left(  \frac{2^{m}k+1-1}{2}\right)  =2^{m-1}k(2^{m}k+1)
\label{odo}%
\end{equation}
so if $x_{0}=2^{m}k+1$ then%
\[
x_{1}=2^{m-1}kx_{0}%
\]

If $m>1$ then $x_{2}$ is just half of $x_{1}$ so that $x_{2}=2^{m-2}kx_{0}$.
By induction%
\begin{equation}
x_{m}=kx_{0} \label{odom}%
\end{equation}

In the special case $k=1$ we obtain an $m$-cycle.

\begin{lemma}
\label{L1}For every positive integer $m$ the recursion (\ref{q1}) has an
$m$-cycle given by the numbers (in the order shown):%
\begin{equation}
2^{m}+1\rightarrow2^{m-1}(2^{m}+1)\rightarrow2^{m-2}(2^{m}+1)\rightarrow
\ldots\rightarrow2(2^{m}+1)\rightarrow2^{m}+1 \label{cyc}%
\end{equation}

\end{lemma}

\medskip

We can infer from (\ref{evo}) that each number of type $2^{l}(2^{m}+1)$
reaches an $m$-cycle of $Q$ in $l$ steps for every $l\in\mathbb{N}$. Are these
the only numbers that reach cycles?

It is reasonable to think that there may be odd numbers other than $2^{m}+1$
that eventually reach cycles. However, we soon see that this is not the case.
The next lemma is the main step in proving this fact.

\begin{lemma}
\label{L2}The following equation has no solution $(\,j,k,m)\in\mathbb{N}^{3}$
where $k\geq3$ is odd:
\begin{equation}
k(2^{j}k+1)=2^{m}+1\label{e1}%
\end{equation}

\end{lemma}

\begin{proof}
By way of contradiction suppose that (\ref{e1}) is true for some
$j,m\in\mathbb{N}$ and some odd $k\geq3$. Rearrange (\ref{e1}) into the
equivalent form%
\begin{equation}
2^{j}k^{2}+k-1=2^{m} \label{e2}%
\end{equation}

Since $k-1$ is even, there are $j_{1},k_{1}\in\mathbb{N}$ with $k_{1}$ odd
such that%
\begin{equation}
k-1=2^{j_{1}}k_{1} \label{e2a}%
\end{equation}

Inserting this in (\ref{e2}) gives%
\begin{equation}
2^{j}k^{2}+2^{j_{1}}k_{1}=2^{m} \label{e3}%
\end{equation}

Clearly $m$ must exceed both $j$ and $j_{1}$. There are three possibilities:

If $j>j_{1}$ then we may divide (\ref{e3}) by $2^{j_{1}}$ to obtain%
\[
2^{j-j_{1}}k^{2}+k_{1}=2^{m-j_{1}}%
\]

But this equality is imposible since its left side is odd while its right side
is even.

Similarly, if $j>j_{1}$ then division by $2^{j}$ yields%
\[
k^{2}+2^{j_{1}-j}k_{1}=2^{m-j}%
\]

This equality is also impossible due to odd-even disparity on opposite sides.

Therefore, (\ref{e3}) may hold only if $j_{1}=j$, in which case after dividing
by $2^{j}$ we obtain the equivalent reduced form:%
\begin{equation}
k^{2}+k_{1}=2^{m-j} \label{e4}%
\end{equation}

This equality may hold for some $m>j$ if $k$ is large enough, since by
(\ref{e2a})%
\[
k=2^{j_{1}}k_{1}+1\geq2^{j}+1
\]

Keep in mind that the equality in (\ref{e2}) may hold only if that in
(\ref{e4}) does or equivalently, if (\ref{e4}) does not hold then neither does
(\ref{e2}).

Next, we use (\ref{e2a}) to transform (\ref{e4}) into the equivalent form that
involves only the odd number $k_{1}$%
\begin{align}
(2^{j}k_{1}+1)^{2}+k_{1}  &  =2^{m-j}\nonumber\\
2^{2j}k_{1}^{2}+2^{j+1}k_{1}+k_{1}+1  &  =2^{m-j} \label{e5}%
\end{align}

Comparing (\ref{e2}) and (\ref{e5}) shows that our argument implies a
reduction of both sides of (\ref{e2}) by a factor of $2^{j}$. Furthermore, the
even number $k_{1}+1$ may be written like (\ref{e2a}) as%
\begin{equation}
k_{1}+1=2^{j_{2}}k_{2}\quad\text{or\quad}k_{1}=2^{j_{2}}k_{2}-1 \label{e5a}%
\end{equation}
where $j_{2},k_{2}\in\mathbb{N}$ with $k_{2}$ odd. This similarity suggests a
reduction process in which (\ref{e5}) can be further reduced by a factor of at
least 2 at each step. For instance, using (\ref{e5a}) in (\ref{e5}) gives%
\[
2^{2j}k_{1}^{2}+2^{j+1}k_{1}+2^{j_{2}}k_{2}=2^{m-j}%
\]
which may hold only if $j_{2}=j+1$. In this case, division by $2^{j+1}$
yields
\begin{equation}
2^{j-1}k_{1}^{2}+k_{1}+k_{2}=2^{m-2j-1} \label{e5b}%
\end{equation}

It is worth mentioning here that this equality is valid only if $j\geq2$
because we cannot have an odd number on the left side and an even one the
right. Thus (\ref{e1}) does not hold for any $k,m$ if $j=1$, a fact that can
be proved independently by factoring the left side of (\ref{e2}) with $j=1$.
Now, we change all $k_{1}$ to $k_{2}$ in (\ref{e5b})%
\begin{align}
2^{j-1}(2^{j+1}k_{2}-1)^{2}+2^{j+1}k_{2}-1+k_{2}  &  =2^{m-2j-1}\nonumber\\
2^{3j+1}k_{2}^{2}-2^{2j+1}k_{2}+2^{j-1}+2^{j+1}k_{2}-1+k_{2}  &
=2^{m-2j-1}\nonumber\\
2^{3j+1}k_{2}^{2}+(-2^{2j+1}+2^{j+1})k_{2}+2^{j-1}+k_{2}-1  &  =2^{m-2j-1}
\label{e6}%
\end{align}

Note that (\ref{e2}) holds only if (\ref{e6}) does. For (\ref{e6}) to hold,
$m$ must exceed $2j+1$ and also
\[
k=2^{j}k_{1}+1=2^{j}(2^{j+1}k_{2}-1)+1\geq2^{j}(2^{j+1}-1)+1
\]

In particular, with $j=2$ (\ref{e1}) has no solutions with $k<2^{2}%
(2^{3}-1)+1=29$ thus ruling out small values for $k$. For larger $j$ the above
lower bound of $k$ increases even more.

We continue by induction: based on (\ref{e6}) we assume that $k,m$ were large
enough that the equation was reduced $i$ times in the above manner and led to
the equality%
\begin{equation}
2^{a_{i1}\,j+a_{i2}}k_{i}^{2}+k_{i}\sum_{l}2^{b_{l1}\,j+b_{l2}}\beta_{l}%
+\sum_{p}2^{c_{p1}\,j+c_{p2}}\gamma_{p}+k_{i}\pm1=2^{m-d_{i1}\,j+d_{i2}}
\label{ie1}%
\end{equation}
where%
\[
a_{i1},b_{l1},c_{p1},d_{i1},l,p\in\mathbb{N},\quad a_{i2},b_{l2},c_{p2}%
,d_{i2}\in\mathbb{Z},\quad\beta_{l},\gamma_{p}\in\{-1,1\}.
\]
and for all $i,j,l,p$
\[
a_{i1}\,j+a_{i2},b_{l1}\,j+b_{l2},c_{p1}\,j+c_{p2}\in\mathbb{N}%
\]

Next, define%
\[
k_{i}\pm1=2^{j_{i+1}}k_{i+1}%
\]
and substitute it in (\ref{ie1}) and arguing as before, we find that
(\ref{ie1}) holds only if%
\[
j_{i+1}=\min_{l,\,p}\{a_{i1}\,j+a_{i2},b_{l1}\,j+b_{l2},c_{p1}\,j+c_{p2}\}
\]
that is, $j_{i+1}$ must be the smallest power of 2 on the left side of
(\ref{ie1}). For instance, in (\ref{e6}) $j_{3}=j-1$ since the least power of
2 on the left side is $j-1$.

The division by $2^{j_{i+1}}$ yields a unit term $\pm1$ depending on the
coefficient of the term $\pm2^{j_{i+1}}$. This unit term then combines with
$k_{i}$ to generate $k_{i+1}$ via%
\[
k_{i}\pm1=2^{j_{i+1}}k_{i+1}%
\]

There is the question as to whether the $\pm2^{j_{i+1}}$ has a non-unit
coefficient in which case the last equality is questionable. But this is not
problematic because if we have $n$ occurrences of $2^{j_{i+1}}$ with $n>1$
then with $j_{i+1}$ being the least power, $n$ must be odd. Thus,%
\[
n2^{j_{i+1}}=(n-1)2^{j_{i+1}}+2^{j_{i+1}}=2^{j_{i+1}+\alpha}\delta+2^{j_{i+1}}%
\]
where $\alpha,\delta\in\mathbb{N}$. A similar argument applies if the
coefficient of $2^{j_{i+1}}$ is negative. It is clear that the term
$2^{j_{i+1}}$ with the least power appears uniquely in (\ref{ie1}).

Next, set $k_{i}=2^{j_{i+1}}k_{i+1}\pm1$ as mentioned above and divide
(\ref{ie1}) by $2^{j_{i+1}}$ to obtain the following:%
\[
2^{a_{i1}\,j+a_{i2}-j_{i+1}}k_{i}^{2}+k_{i}\sum_{l}2^{b_{l1}\,j+b_{l2}%
-j_{i+1}}\beta_{l}+\sum_{p}2^{c_{p1}\,j+c_{p2}-j_{i+1}}\gamma_{p}%
+k_{i+1}=2^{m-d_{i1}\,j+d_{i2}-j_{i+1}}%
\]

The right hand side of the above equation may be transformed to an equation
involving only $k_{i+1}$%
\begin{gather*}
2^{a_{i1}\,j+a_{i2}-j_{i+1}}(2^{j_{i+1}}k_{i+1}\pm1)^{2}+(2^{j_{i+1}}%
k_{i+1}\pm1)\sum_{l}2^{b_{l1}\,j+b_{l2}-j_{i+1}}\beta_{l}+\\
\hspace{3.5in}+\sum_{p}2^{c_{p1}\,j+c_{p2}-j_{i+1}}\gamma_{p}+k_{i+1}=\\
2^{a_{i1}\,j+a_{i2}-j_{i+1}}(2^{2j_{i+1}}k_{i+1}^{2}\pm2^{j_{i+1}+1}%
k_{i+1}+1)+2^{j_{i+1}}k_{i+1}\sum_{l}2^{b_{l1}\,j+b_{l2}-j_{i+1}}\beta_{l}+\\
\hspace{2in}\pm\sum_{l}2^{b_{l1}\,j+b_{l2}-j_{i+1}}\beta_{l}+\sum_{p}%
2^{c_{p1}\,j+c_{p2}-j_{i+1}}\gamma_{p}+k_{i+1}=\\
2^{a_{i1}\,j+a_{i2}+j_{i+1}}k_{i+1}^{2}+\left(  \pm2^{a_{i1}\,j+a_{i2}+1}%
+\sum_{l}2^{b_{l1}\,j+b_{l2}}\beta_{l}\right)  k_{i+1}+2^{a_{i1}%
\,j+a_{i2}-j_{i+1}}+\\
\hspace{2in}\pm\sum_{l}2^{b_{l1}\,j+b_{l2}-j_{i+1}}\beta_{l}+\sum_{p}%
2^{c_{p1}\,j+c_{p2}-j_{i+1}}\gamma_{p}+k_{i+1}=%
\end{gather*}

The last three terms prior to $k_{i+1}$ may now be grouped together and
written as a single summation. Doing this and some minor rewriting,%
\begin{gather}
2^{(a_{i1}+1)\,j+(a_{i2}+j_{i+1}-j)}k_{i+1}^{2}+\left(  \pm2^{a_{i1}%
\,j+a_{i2}+1}+\sum_{l}2^{b_{l1}\,j+b_{l2}}\beta_{l}\right)  k_{i+1}+\sum
_{q}2^{h_{q1}\,j+h_{q2}}\eta_{q}+k_{i+1}\label{ie2}\\
\hspace{4.5in}=2^{m-d_{i1}\,j+d_{i2}-j_{i+1}}\nonumber
\end{gather}

Comparing this equation with (\ref{ie1}) indicates that the unit term $\pm1$
is missing in explicit form but it is there implicitly (and uniquely) since
one of the powers of 2 in (\ref{ie2}) reduces to 0 after subtracting $j_{i+1}$
from its exponent. We conclude that the process of reducing the power on the
two sides continues indefinitely and (\ref{e1}) does not have a solution
$j,k,m$ as claimed.
\end{proof}

\medskip

The next theorem substantially improves a result in \cite{HS2} by giving a
complete qualitative characterization of all orbits of the divide-or-choose-2 rule.

\begin{theorem}
\label{T}For every $m,l\in\mathbb{N}$ an orbit of the recursion (\ref{q1})
with initial value $x_{0}=2^{l}(2^{m}+1)$ ends in a cycle of length $m$. For
every other initial value the orbit goes to infinity.
\end{theorem}

\begin{proof}
Lemma \ref{L1} and the related discussion above prove the first statement. To
prove the second statement it is enough to consider $x_{0}=2^{j_{0}}k_{0}+1$
where $k_{0}$ is odd and $j_{0}\in\mathbb{N}$ because any even initial value
reaches an odd number by successive divisions. If $k_{0}=1$ then the orbit
ends in a $j_{0}$-cycle. If $k_{0}\geq3$ then%
\[
x_{j_{0}}=k_{0}x_{0}\geq3x_{0}%
\]
so on the down-swing the orbit reaches a larger odd number $x_{j_{0}}$. Let
$x_{j_{0}}=2^{j_{1}}k_{1}+1$ where $k_{1}$ is odd and note that%
\begin{equation}
k_{0}(2^{j_{0}}k_{0}+1)=k_{0}x_{0}=x_{j_{0}}=2^{j_{1}}k_{1}+1 \label{Te1}%
\end{equation}

Lemma \ref{L2} implies that this equation has no solutions $j_{1},j_{0},k_{0}$
with $k_{1}=1$ so it follows that $k_{1}\geq3$. In fact, a given odd number
$x_{0}$ determines $j_{0},k_{0}$ uniquely and from (\ref{Te1}) we find $j_{1}$
as the largest positive integer such that $[k_{0}(2^{j_{0}}k_{0}%
+1)-1]2^{-j_{1}}$ is an integer that we label $k_{1}$. The numbers $j_{1}$ and
$k_{1}$ exist with $k_{1}$ odd because $j_{1}$ is maximal and $k_{0}(2^{j_{0}%
}k_{0}+1)$ is odd and not equal to 1.

Now, the arguments that led to (\ref{Te1}) can be repeated to yield:%
\[
x_{j_{0}+j_{1}}=k_{1}x_{j_{0}}=k_{1}k_{0}x_{0}\geq3^{2}x_{0}%
\]

Setting $x_{j_{0}+j_{1}}=2^{j_{2}}k_{2}+1$ and arguing as before, it follows
that $k_{2}\geq3$ and the process continues. A simple induction verifies that
after $n$ steps%
\[
x_{j_{0}+j_{1}+\cdots+j_{n}}=k_{n}\cdots k_{1}k_{0}x_{0}\geq3^{n+1}x_{0}%
\]

Note that the subsequence or sub-orbit consisting of the numbers
\[
x_{0}, x_{j_{0}}, x_{j_{0}+j_{1}}, \ldots, x_{j_{0}+j_{1}+\cdots+j_{n}},
\ldots
\]
are precisely the odd terms of the orbit. They are separated by decreasing
chains of even terms (except when $j_{p}=1$ for some $p$ so we have
consecutive odd terms $x_{j_{0}+\cdots+j_{p-1}}$ and $x_{j_{0}+\cdots
+j_{p-1}+j_{p}}$). It follows that the orbit goes to infinity if $k_{0}>1$ in
$x_{0}$.
\end{proof}

\medskip

An obvious consequence of Theorem \ref{T} is that all \textit{unbounded}
orbits of the divide-or-choose-2 rule must converge to infinity; they cannot
have bounded subsequences. Also most orbits go to infinity in the sense that
most odd numbers are of type $2^{j}k+1$ with $k\geq3.$ For example, if
$x_{0}=2^{m}-1$ with $m\geq3$ then we can write it as%
\[
x_{0}=2(2^{m-1}-1)+1
\]
with $j=1$ and $k=2^{m-1}-1\geq3$ and infer via Theorem \ref{T} that the
corresponding orbit goes to infinity. Note that numerical simulation on a
computer may generate a false orbit that appears to be bounded, if due to
finite memory the computer fails at some step of the iteration to distinguish
between the numbers $2^{j}k\pm1$ and $2^{j}k$ for large $j,k$.

Similarly,
\[
x_{0}=2^{m}+3=2(2^{m-1}+1)+1
\]
so if $m\geq2$ then the corresponding orbit goes to infinity. These statements
were left as open questions in \cite{HS2}.

\medskip



\begin{thebibliography}{9}                                                                                                %


\bibitem {LG}Lagarias, J.C., (Editor) \textit{The Ultimate Challenge: The 3x+1
Problem}, American Mathematical Society, Providence, 2010


\bibitem {HS2}Sedaghat, H. (2020) Bounded orbits of quadratic Collatz-type
recursions, arxiv:2004.07357
\end{thebibliography}
\end{document}